\documentclass[11pt]{amsart}
\usepackage{standard}

\usepackage{amsrefs}
\DeclareSymbolFont{bbold}{U}{bbold}{m}{n}
\DeclareSymbolFontAlphabet{\mathbbold}{bbold}

\newcommand{\tOmega}{{\Omega_0}}
\newcommand{\fcal}{\mathcal{F}}
\newcommand{\torus}{\mathbb{T}}
\newcommand{\bu}{\mathbf{u}}
\newcommand{\bff}{\mathbf{f}}
\DeclareMathOperator{\spec}{spec}
\DeclareMathOperator{\Id}{Id}
\newcommand{\HH}{\mathcal{H}}

\newtheorem*{thm}{Theorem}

\title{Periodic damping gives polynomial energy decay}
\author{Jared Wunsch}
\address{Department of Mathematics\\ Northwestern University\\2033 Sheridan Rd.\\ Evanston IL 60208\\USA}
\email{jwunsch@math.northwestern.edu}

\thanks{The author is grateful to Nicolas Burq for helpful discussions
leading to these results, as well as to Claude Zuily for his comments
on the manuscript and to Julien Royer and Satbir Malhi for pointing out errors in the
first version.  The paper also benefited from the helpful comments of
two anonymous referees.
Partial support was provided by NSF grant
DMS--1265568.  The author also gratefully acknowledges the hospitality
of the Institut Henri Poincar\'e, where this work was carried out.}

\date{\today}

\begin{document}

\maketitle
\begin{abstract}
Let $u$ solve the damped Klein--Gordon equation
$$
\big( \pa_t^2-\sum \pa_{x_j}^2 +m\Id +\gamma(x) \pa_t \big) u=0
$$
on $\RR^n$ with $m>0$ and $\gamma\geq 0$ bounded below on a $2 \pi
\ZZ^n$-invariant open set by a positive constant.  We show that the
energy of a solution decays at a polynomial rate.  This is proved via
a periodic observability estimate on $\RR^n.$
\end{abstract}

\section{Introduction}

Consider the damped Klein--Gordon equation on $[0,\infty) \times \RR^n:$
\begin{equation}\label{DWE}
\big( \pa_t^2-\sum \pa_{x_j}^2 +m\Id +\gamma(x) \pa_t \big) u=0,
\end{equation}
with $\gamma (x) \geq 0$ for all $x,$ and $m>0.$ Burq--Joly
\cite{BuJo:14} have recently proved that if there is \emph{uniform
  geometric control} in the sense that there exist $T, \ep>0$ such
that $\int_0^T \gamma(x(t))\, dt \geq \ep$ along every straight line
unit-speed trajectory, then $u$ enjoys exponential energy decay, thus
generalizing classic results of Bardos, Lebeau, Rauch, and Taylor
\cite{RaTa:74}, \cite{MR90h:35030}, \cite{BaLeRa:92} to a noncompact
setting.  By contrast, in the case of merely periodic $\gamma$ (or,
more generally, under the assumption that $\gamma$ is strictly
positive on a family of balls whose dilates cover $\RR^n$) then
Burq--Joly show that a logarithmic decay of energy still holds.

In this note, we show that in fact $u$ enjoys at least a \emph{polynomial} rate
of energy decay (with derivative loss) provided that $\gamma$ is nontrivial and periodic, or,
more generally, strictly positive on a periodic set:
\begin{theorem}\label{theorem:decay}
Assume that $m>0$ and $0 \leq \gamma \in L^\infty$ and that there exist
$\ep>0$ and a $2\pi \ZZ^n$-invariant open set $\Omega
\subset \RR^n$ such that $\gamma(x) \geq \ep$ for a.e.\ $x \in \Omega.$
Then there exists $C>0$ such that for $u$ solving \eqref{DWE},
$$
\norm{(u(t), u_t(t))}_{H^1\times L^2} \leq \frac {C}{\sqrt{1+t}}\norm{(u(0), u_t(0))}_{H^2\times H^1}.
$$
\end{theorem}
Note that we do not require any hypothesis of geometric control.

We proceed by a standard route to this estimate by first proving an
\emph{observability estimate}, which then leads to a
\emph{resolvent estimate}.

  Let
$$
\Lap=-\sum \pa_{x_j}^2
$$
denote the nonnegative Laplace operator.
The
observability estimate (which may be of independent interest
owing to applications in control theory) is then as follows:
\begin{theorem}\label{theorem:obs}
Let $\Omega \subset \RR^n$ be a nonempty, open, $2 \pi \ZZ^n$-invariant
set.
For all $\lambda \in \RR$ we have the following estimate:
$$
(\Lap-\lambda) u=f \Longrightarrow \norm{u}_{L^2(\RR^n)} \leq C\big( \norm{f}_{L^2(\RR^n)}+\norm{u}_{L^2(\Omega)} \big)
$$
\emph{with $C$ independent of $\lambda$.}
\end{theorem}

From the observability estimate, it is not difficult to obtain a
\emph{resolvent estimate} as follows:
\begin{theorem}\label{theorem:resolvent}
Let $\gamma$ be as in Theorem~\ref{theorem:decay}.  Then
\begin{equation}\label{quasimode}
(\Lap+m\Id +is \gamma (x)-s^2\Id) u=f \Longrightarrow \norm{u}_{L^2}
\lesssim C (1+\smallabs{s}) \norm{f}_{L^2}.
\end{equation}
\end{theorem}

The strategy will be to prove Theorem~\ref{theorem:obs} by reducing it
to known observability estimates on the torus;
this argument is the main novelty here.  This leads to
Theorem~\ref{theorem:resolvent} by standard arguments (given below).
The decay estimate Theorem~\ref{theorem:decay} then follows by a
functional-analytic argument due to Borichev--Tomilov \cite{BoTo:10}.

We remark that \emph{the decay rate obtained here is almost certainly
  not optimal} in the case of smooth damping.  The work of
Anantharaman-L\'eautaud \cite{AnLe:2014} shows that on $\torus^2$ one
can obtain better estimates for damping with better than $L^\infty$
smoothness by moving the damping to the left-hand side of the
estimates, treating it as part of the operator rather than as a term
to be estimated as an error.  The arguments used here make it
difficult to bring to bear the finer results of \cite{AnLe:2014} in
the periodic setting, as the incorporation of the damping term into
the operator one is trying to estimate makes the proof of
our Proposition~\ref{proposition:observe} fail badly.  To obtain a
stronger estimate, one would therefore have to follow the
second-microlocal arguments of \cite{AnLe:2014} directly rather than
simply using the resulting estimate on the torus.  As refined
estimates linked to the regularity of the damping term remain a
difficult subject of current research, we will not pursue such an
approach in this note.

In what follows, the constant $C$ will change from line to line, but
will always be independent of the spectral parameter.  As noted above,
$\Lap\geq 0$ denotes the nonnegative Laplacian, and we also denote
$D_{x_j}=i^{-1}\pa_{x_j},$ so that $\Lap=\sum D_{x_j}^2,$ or, abusing
notation slightly, $\Lap=D_x^2.$  With no subscript, the notation
$\norm{\bullet}$ denotes $L^2$ norm.  We use the notation for the
standard torus $\torus^n
\equiv \RR^n/2 \pi \ZZ^n.$

\section{Proofs of Main results}
\subsection{Twisted Laplacian}
We begin by establishing observability estimates on a bundle
Laplacian on the flat torus (cf.\ Lemma~2.4 of \cite{BuZw:12} for a
related estimate).

Let $\alpha \in \RR^n.$
Set \begin{equation}\label{Halpha} H_\alpha=(D_x
  -\alpha)^2.\end{equation}
Note that these operators are all self-adjoint with the same domain
independent of $\alpha.$

\begin{proposition}\label{proposition:observe}
Let $\Upsilon\subset \torus^n$ be open and nonempty.
For all $\alpha \in [0,1)^n,$
\begin{equation}\label{resolventobservability}
(H_\alpha-\lambda ) u=f\ \text{on } \torus^n 
\Longrightarrow
\norm{u}_{L^2(\torus^n)} \leq C {\norm{f}}_{L^2(\torus^n)} + C {\norm{u}_{L^2(\Upsilon)}},
\end{equation}
\emph{with constants independent of $\alpha$
  and  $\lambda\in \RR.$}
\end{proposition}
\begin{proof}
With $\alpha=0$ the result is known, from the estimates of
Jaffard \cite{Ja:90} in dimension $n=2$
and Komornik \cite{Ko:92} in higher dimension.  We will use these
results to generalize to variable $\alpha$.

We recall that one approach
to proving estimates of the form \eqref{resolventobservability} is based on
an {observability estimate} for the the Schr\"odinger propagator (see Theorem~4 of
\cite{AnMa:14}): we say that \emph{Schr\"odinger observability} holds for
$H_\alpha$ if
for every open, nonempty $\omega \subset \torus^n$ and every $T>0$
there exists $C=C(T,\omega)$ such that
\begin{equation}\label{sch.obs}
\norm{f}^2 \leq C \int_0^T \norm{e^{it H_\alpha} f}_{L^2(\omega)}^2 \, dt.
\end{equation}
That Schr\"odinger observability \eqref{sch.obs} for $H_\alpha$ is equivalent to the
resolvent estimate \eqref{resolventobservability} for $H_\alpha$
follows from Theorem 5.1 of Miller~\cite{Mi:05}.

Thus we will prove the proposition by proving Schr\"odinger observability for
any $\alpha \in \RR^n.$ Given an open $\Upsilon\subset \torus^n,$ fix a
nonempty open $\omega\subset
\torus^n$ such that $\overline{\omega} \subset \Upsilon,$ hence there
exists $T>0$ such that $x \in \omega$ and $d(x,y)<4\sqrt{n} T$ implies $y \in
\Upsilon.$ (Here $d(x,y)$ denotes the distance function between points on $\torus^n.$)

Now we note that the propagators $$U_\alpha(t) \equiv e^{it H_\alpha}$$ all
commute with one another, and indeed we may factor
\begin{align}
U_\alpha(t) &=U_0(t) \exp(-2it \alpha\cdot D
+it\smallabs{\alpha}^2)\\ &=e^{it\smallabs{\alpha}^2} \tau_{-2 t\alpha} U_0(t),
\end{align}
where for $\theta \in \RR^n,$ $\tau_\theta$ denotes the translation
operator $\tau_\theta f(x) =f(x+\theta).$
Thus, by $H_0$-observability and the choice of $T \ll 1$ so that
$\tau_{2 t \alpha}(\omega) \subset \Upsilon$ for $t \in [0,T],$ we obtain
\begin{align}
\norm{f}^2 &\leq C \int_0^T \norm{U_0(t) f}_{L^2(\omega)}^2 \,
dt\\
&\leq C \int_0^T \norm{ e^{it \smallabs{\alpha}^2} \tau_{-2\alpha t} U_0(t) f}_{L^2(\tau_{2 \alpha t}(\omega))}^2 \,
dt\\
&\leq C \int_0^T \norm{ e^{it \smallabs{\alpha}^2} \tau_{-2\alpha t} U_0(t) f}_{L^2(\Upsilon)}^2 \,
dt\\
&\leq C \int_0^T \norm{U_\alpha(t) f}_{L^2(\Upsilon)}^2 \,
dt.
\end{align}
As noted above, Theorem~5.1 of
\cite{Mi:05} now shows that this Schr\"odinger observability estimate
implies our resolvent estimate.
\end{proof}

\subsection{Observability estimate}

The proof of Theorem~\ref{theorem:obs} now proceeds as follows.  For
$g\in \smallang{x}^{-s} H^{-\infty}(\RR^n)$ with $s>n/2,$ define its periodization $\Pi
g\in \mathcal{D}' (\torus^n)$ by
$$
\Pi g(x) =\sum_{\ell \in \ZZ^n} g(x+2\pi \ell).
$$
More generally, for $\alpha \in \RR^n$ we set
$$
(\Pi_\alpha g) = \Pi (e^{i \alpha x} g).
$$
Note that this quantity is quasi-$\ZZ^n$-periodic in $\alpha\in \RR^n$: we
have for $k \in \ZZ^n,$
$$
(\Pi_{\alpha+k} g)(x) = e^{ikx} (\Pi_\alpha g)(x).
$$
\begin{lemma}\label{lemma:plancherel}
We have the equality of $L^2$ norms \begin{equation}\label{poisson}
\norm{g}^2_{L^2 (\RR^n)} =\int_{[0,1)^n} \norm{\Pi_\alpha g}^2_{L^2(\torus^n)}\, d\alpha.
\end{equation}
More generally, if $\Omega \subset \RR^n$ is $2\pi \ZZ^n$-invariant and
$\tOmega$ denotes its projection to $\torus^n,$
$$
\norm{g}^2_{L^2 (\Omega)} =\int_{[0,1)^n} \norm{\Pi_\alpha g}^2_{L^2(\tOmega)}\, d\alpha.
$$
\end{lemma}
\begin{proof}
We use Fubini to compute the Fourier coefficients of the periodic
functions $\Pi_\alpha g$ on $\torus^n:$
\begin{align*}
\widehat{\Pi_\alpha g}(\ell) &= (2\pi)^{-n/2}\int_{[0,2\pi]^n} \sum_{m \in \ZZ^n} g(x+2 \pi m
) e^{i\alpha (x+2\pi m)} e^{-i\ell x} \, dx\\
 &= (2\pi)^{-n/2}\int_{[0,2\pi]^n} \sum_{m \in \ZZ^n} g(x+2 \pi m
) e^{i\alpha (x+2\pi m)} e^{-i\ell (x+ 2 \pi m)} \, dx\\
&= (2\pi)^{-n/2}\int_{\RR^n} g(y) e^{i\alpha y} e^{-i\ell y} \, dy\\
&= \fcal(g)(\ell-\alpha).
\end{align*}
Integrating the sum of squares of the RHS over the unit cube gives $\norm{g}^2_{L^2(\RR^n)}$
by Fubini and Plancherel on $\RR^n,$
while on the LHS, we get
$$
\int_{[0,1)^n} \sum \abs{\widehat{\Pi_\alpha g}(\ell)}^2\, d\alpha =\int_{[0,1)^n}
\norm{{\Pi_\alpha g}}^2_{L^2(\torus^n)}\, d\alpha
$$
by Plancherel on the torus.

The generalization to taking the norm over $\Omega$ is proved simply
by applying \eqref{poisson} to the function $\mathbbold{1}_{\Omega} g.$
\end{proof}

Now we note that
\begin{align*}
H_\alpha-\lambda &\equiv  (D_x-\alpha)^2 -\lambda\\
&=  e^{i\alpha x} (\Lap-\lambda )e^{-i\alpha x}.
\end{align*}
Thus $(\Delta-\lambda)u=f$ yields
$$
(H_\alpha-\lambda) e^{i\alpha x} u=e^{i\alpha x} f\quad \text{on } \RR^n.
$$
Applying $\Pi$ to both sides and using translation-invariance of
$H_\alpha$, we get an equation on the torus:
$$
(H_\alpha-\lambda) (\Pi_\alpha u) =\Pi_\alpha f \quad \text{on } \torus^n.
$$
Applying Proposition~\ref{proposition:observe}, we obtain for
every $\alpha$ in a fundamental domain (and with constants independent of $\alpha$)
$$
\norm{\Pi_\alpha u} ^2\leq C {\norm{\Pi_\alpha f}}^2 + C {\norm{\Pi_\alpha u}^2_{L^2(\tOmega)}},
$$
Now by Lemma~\ref{lemma:plancherel} we may integrate both sides in $\alpha \in [0,1)^n$ to obtain
$$
\norm{u} ^2\leq C {\norm{f}}^2 + C {\norm{u}^2_{L^2(\Omega)}}.
$$
This concludes the proof of Theorem~\ref{theorem:obs}.\qed

\subsection{Resolvent estimate}
We now prove Theorem~\ref{theorem:resolvent}.  We will be brief, as
this is ground well-trodden by other authors.

We start by noting that if we pair the equation \eqref{quasimode} with
$u$ and take the real part, we obtain (using Cauchy-Schwarz) for
$\smallabs{s}\leq s_0 \equiv \sqrt{m}/2$
$$
\norm{u}^2_{H^1(\RR^n)} \leq C \norm{f}^2_{L^2(\RR^n)}.
$$
This proves the estimate near $s=0,$ so we will take
$\smallabs{s}>s_0$ for fixed $s_0$ below.

Again pairing \eqref{quasimode} with $u$ and this time taking the imaginary part yields the
usual estimate
$$
\norm{\sqrt\gamma u}^2 \leq \frac{C}{\smallabs{s}} \norm{f}\norm{u}.
$$
On the other hand, applying
 Theorem~\ref{theorem:obs} to \eqref{quasimode} with the damping term
 on the right-hand side (and $\lambda=s^2-m$) yields
\begin{equation}\label{mainagain}\begin{aligned}
\norm{u} &\leq C \norm{f} + C \smallabs{s} \norm{\gamma u} +C
\norm{u}_{L^2(\Omega)}\\&\leq C \norm{f} + C \smallabs{s} \norm{\gamma u},
\end{aligned}
\end{equation}
where we chose $\Omega$ contained in the set where $\gamma\geq \ep$ a.e.\ for
some $\ep>0$ (and used $s\geq s_0$).  Combining these estimates and observing that $\gamma
\leq  C \sqrt \gamma$ a.e.\  yields for $\smallabs{s}\geq s_0$
$$
\norm{\sqrt\gamma u}^2 \leq \frac{C}{\smallabs{s}} \norm{f}^2+ C  \norm{f}
\norm{\sqrt\gamma u}.
$$
Applying Cauchy-Schwarz we obtain
$$
\norm{\sqrt\gamma u}^2\leq C\norm{f}^2,\quad \smallabs{s}>s_0.
$$
Finally returning to \eqref{mainagain} gives
$$
\norm{u} \leq C \norm{f} + C \smallabs{s} \norm{f}.\qed
$$

\subsection{Proof of energy decay}
In this section, we apply the resolvent estimate,
Theorem~\ref{theorem:resolvent}, to prove our result on energy decay
for the damped Klein--Gordon equation, Theorem~\ref{theorem:decay}.  To
do this we follow the strategy used by Anantharaman--L\'eautaud
\cite{AnLe:2014}, albeit in the much simpler framework of
\cite{BuJo:14}, in which low energy issues are rendered moot by the
positive Klein--Gordon mass.  (We cannot simply quote Proposition~2.4
of \cite{AnLe:2014} verbatim, however, as its hypotheses include a
compact resolvent assumption that fails here.)

The strategy consists of employing the
following theorem of Borichev--Tomilov \cite{BoTo:10} (this is in fact
just one part of Theorem 2.4 of \cite{BoTo:10}):
\begin{thm}[Borichev--Tomilov]
Let $e^{tA}$ be a bounded $\mathcal{C}^0$ semigroup on a Hilbert space with generator $A$ with $\spec(A)
\cap i \RR =\emptyset.$  Then 
$$
\norm{(A-is\Id)^{-1}}=O(\smallabs{s}^\alpha),\ \smallabs{s}\to \infty
\Longleftrightarrow \norm{e^{tA}A^{-1}} =O(t^{-1/\alpha}),\ t \to \infty.
$$
\end{thm}
(This represents a slight strengthening of a prior result of
Batty-Duyckaerts \cite{BaDu:08} in Banach spaces.)

We will apply this theorem to the semigroup generated by
$$
A= \begin{pmatrix}
0 & \Id \\
-\Lap-m\Id & -\gamma
\end{pmatrix} 
$$
acting on the energy space $\HH\equiv H^1(\RR^n) \times L^2 (\RR^n).$

The resolvent estimate from Theorem~\ref{theorem:resolvent} implies
the condition on non-imaginary spectrum on $A$ as well as the
resolvent estimate on $A$ as follows: if $\bu=(u_0, u_1)^t$ and
$\bff=(f_0,f_1)^t$ then
$$
(A-is \Id) \bu =\bff
$$
is equivalent to
\begin{equation}\label{system}
\begin{aligned}
(\Lap+m+i s \gamma -s^2) u_0 &=f_1+ (\gamma+is) f_0,\\
u_1 &= f_0 + is u_0,
\end{aligned}
\end{equation}
i.e., if we let $R(is)$ denote the inverse of $(\Lap+m+is \gamma-s^2),$
we have
$$
\bu = (A-is \Id)^{-1} \bff =\begin{pmatrix}
R(is)(\gamma+is) & R(is)\\
\Id + R(is) (is)(\gamma+is) & is R(is)
\end{pmatrix} \bff.
$$
Existence of $R(is)$ on $L^2$ (with norm $O(\smallang{s})$) is Theorem~\ref{theorem:resolvent}, and pairing
\eqref{quasimode} with $u$ as usual and taking real parts easily
establishes that $R(is): L^2 \to H^1$ with norm $O(\smallang{s}^2).$
Thus $(A-is \Id)$ is invertible and we have verified
the spectral condition.  We can further use these methods to estimate
$$
\norm{(A-is \Id)^{-1}}_{\HH\to \HH}=O(\smallang{s}^2);
$$
details of the argument
can be found in, e.g., Lemma~4.6 of \cite{AnLe:2014} (cf.\ also
\cite{MR95i:58175} and \cite{BuHi:07}).  This yields the decay rate
$\smallang{t}^{-1/2}$ for the damped Klein Gordon equation by the
theorem of Borichev--Tomilov.\qed

\bibliographystyle{plain}
\bibliography{all}

\end{document}